\documentclass[reqno,12pt]{amsart}

\usepackage[shortlabels]{enumitem}
\usepackage[margin=1in]{geometry}
\usepackage{ifpdf}
\usepackage{amsmath}
\usepackage{amsfonts}
\usepackage{amssymb}
\usepackage{amsthm}
\usepackage[hyperfootnotes=false]{hyperref}
\usepackage{setspace}
\usepackage{amsrefs}
\usepackage{nicefrac}
\usepackage{graphicx}

\newcommand{\ignore}[1]{}

\renewcommand{\Im}{\operatorname{Im}}

\newcommand{\abs}[1]{\left\lvert {#1} \right\rvert}

\newcommand{\snorm}[1]{\lVert {#1} \rVert}

\newcommand{\C}{{\mathbb{C}}}
\newcommand{\R}{{\mathbb{R}}}

\newcommand{\bP}{{\mathbb{P}}}

\newcommand{\sM}{{\mathcal{M}}}

\newcommand{\sT}{{\mathcal{T}}}

\newcommand{\sX}{{\mathcal{X}}}

\newtheorem{thm}{Theorem}[section]

\newtheorem{prop}[thm]{Proposition}

\newtheorem{cor}[thm]{Corollary}

\newtheorem{lemma}[thm]{Lemma}

\theoremstyle{definition}

\newtheorem{example}[thm]{Example}

\theoremstyle{remark}

\author{Ji\v{r}\'{\i} Lebl}
\thanks{The first author was in part supported by NSF grant DMS-1362337 and
Oklahoma State University's DIG and ASR grants.}
\address{Department of Mathematics, Oklahoma State University,
Stillwater, OK 74078, USA}
\email{lebl@math.okstate.edu}
\date{December 17, 2014}

\ifpdf
\hypersetup{
  pdftitle={Singular Levi-flat hypersurfaces in complex projective space induced by curves in the Grassmannian},
  pdfauthor={Jiri Lebl},
  pdfsubject={Several Complex Variables, CR Geometry},
  pdfkeywords={Levi-flat, complex projective space, nonalgebraic singular hypersurface},
}
\fi

\title%
[Singular Levi-flat hypersurfaces in $\bP^n$ induced by curves in the Grasmannian]%
{Singular Levi-flat hypersurfaces in complex projective space induced by curves in the Grassmannian}

\begin{document}


\begin{abstract}
Let $H \subset {\mathbb P}^n$ be a real-analytic subvariety
of codimension one
induced by a real-analytic curve in the Grassmannian $G(n+1,n)$.
Assuming $H$ has a global defining function,
we prove $H$ is Levi-flat, the closure of
its smooth points of top dimension is a union of complex hyperplanes,
and its singular set is either of dimension $2n-2$
or dimension $2n-4$.
If the singular set is of dimension
$2n-4$, then we show the hypersurface is algebraic
and the Levi-foliation extends to a singular holomorphic
foliation of ${\mathbb P}^n$ with a meromorphic (rational of degree 1)
first integral.  In this case, $H$ is in some sense simply a complex cone over
an algebraic curve in ${\mathbb P}^1$.
Similarly if $H$ has a degenerate singularity, then $H$ is also algebraic.
If the dimension of the singular set is $2n-2$ and is nondegenerate, we show
by construction that the hypersurface need not be algebraic nor
semialgebraic.
We construct a Levi-flat real-analytic subvariety in ${\mathbb P}^2$ of
real codimension 1 with compact leaves that is not contained in any proper
real-algebraic subvariety of ${\mathbb P}^2$.  Therefore a straightforward
analogue of Chow's theorem for Levi-flat hypersurfaces does not hold.
\end{abstract}

\maketitle



\section{Introduction} \label{section:intro}

Let $\bP^n$ be the $n$-dimensional complex projective space, the
space of complex lines through the origin in $\C^{n+1}$.
Chow's theorem \cite{Chow} tells us that any complex analytic subvariety of $\bP^n$
is algebraic.  We naturally ask if Chow's theorem extends to
singular Levi-flat hypersurfaces.  A Levi-flat hypersurface is,
after all, essentially a family of complex submanifolds.
In $\bP^1$ the answer is trivially no.
In $\bP^n$, $n \geq 2$, the question is significantly more difficult.
To this end we study real-analytic subvarieties induced by a curve in
the Grassmannian $G(n+1,n)$, the space of complex hyperplanes in $\bP^n$.
Such subvarieties are possibly the simplest
examples of singular Levi-flat hypersurfaces in $\bP^n$.
While for certain hypersurfaces of this type Chow's theorem holds, we show it does
not hold in general.

A smooth real hypersurface is
Levi-flat if it is pseudoconvex from both sides.  When the hypersurface is
real-analytic, then 
there exist local holomorphic coordinates $z$ such that
the hypersurface is given by
\begin{equation}
\Im z_1 = 0 .
\end{equation}
For any fixed $t \in \R$, 
$\{ z_1 = t \}$ defines a complex hypersurface, 
and these hypersurfaces
give a real-analytic foliation called the
\emph{Levi-foliation}.

Let $H \subset \bP^n$ be a real-analytic subvariety of codimension 1
(dimension $2n-1$).
Let $H_s$ be the \emph{singular} locus of $H$, that is the set of points
where $H$ fails to be a real-analytic submanifold.  Let $H^*$ be
the set of points near which $H$ is a real-analytic submanifold of dimension
$2n-1$.
We say $H$ is \emph{Levi-flat} if $H^*$ is Levi-flat.
When a leaf of the Levi-foliation is closed in $H^*$,
we consider the closure in $\bP^n$ of this leaf as the leaf itself,
and we say the leaf is \emph{compact}.
A compact leaf is a complex analytic
subvariety of $\bP^n$, and hence by Chow's theorem algebraic.
In our case, leaves will be complex hyperplanes in $\bP^n$.

Singular Levi-flat hypersurfaces 
have been the subject of much recent 
interest.
Such hypersurfaces were first studied by
Bedford~\cite{Bedford:flat} and
Burns and Gong~\cite{burnsgong:flat}.
The study of the singular set is related to the problem
of extending the Levi-foliation to a singular holomorphic foliation
of a neighbourhood, or more generally to a holomorphic $k$-web, see
Brunella~\cites{Brunella:lf,Brunella:firstint}, 
Fern{\'a}ndez-P{\'e}rez~\cite{Perez:generic}, 
Cerveau and Lins Neto~\cite{CerveauLinsNeto},
Shafikov and Sukhov~\cite{ShafikovSukhov},
and the author~\cite{Lebl:lfsing}.
A singular holomorphic foliation of codimension one
is given locally by an integrable holomorphic one-form $\omega$, that is,
$d \omega \wedge \omega = 0$.  The integral manifolds are the leaves of
the foliation, and the set where $\omega$ vanishes is the singular set
of the foliation.
A holomorphic foliation extends the Levi-foliation if the leaves of the
two foliations agree as germs at points that are regular points for both
the foliation and $H$.

In projective space, these hypersurfaces have
been studied by
the author~\cite{Lebl:projlf}, 
Fern{\'a}ndez-P{\'e}rez~\cite{Perez:projlf}, and
Fern{\'a}ndez-P{\'e}rez, Mol and Rosas~\cite{Perez:dynPn}.
A related question is the nonexistence of smooth Levi-flat
hypersurfaces in $\bP^n$, a line of research that initiated by
Lins Neto \cite{linsneto:note} in the real-analytic case for $n \geq 3$.
The technique of Lins Neto centers on extending the foliation to all of
$\bP^n$.

The first instinct to construct a singular Levi-flat hypersurface in $\bP^n$ is to
take a real one dimensional curve in the Grassmannian $G(n+1,n)$, which
is the set of complex hyperplanes in $\bP^n$.  If the induced set is
a subvariety, we have a Levi-flat hypersurface.  We will construct such an
example in $\bP^2$.  One might think it
should be easy to take a nonalgebraic real curve and obtain a
nonalgebraic hypersurface $H \subset \bP^n$, but in general the induced set
$H$ fails to be a subvariety or even a semianalytic set.

Let us first characterize those
hypersurfaces induced by a curve in the Grassmannian.
Let $C \subset G(n+1,n)$.  We say a real-analytic subvariety
$H \subset \bP^n$ is induced by $C$ if $H$ is the smallest closed
real-analytic subvariety of $\bP^n$ containing all the complex hyperplanes
corresponding to points in $C$.
It follows from Proposition~\ref{prop:Csubvar} that
the set of complex hyperplanes in $H$ is a closed subvariety
of $G(n+1,n)$.  We therefore lose no generality in assuming $C$ is a closed
real-analytic subvariety of dimension 1 (a curve).




\begin{thm} \label{thm:thm1}
Suppose $H \subset \bP^n$ is a real-analytic subvariety of
real codimension one
induced by an irreducible real-analytic curve in $G(n+1,n)$.  Suppose there
exists a nontrivial real-analytic function $r$ defined in a neighbourhood
$V$ of $H$, such that $r$ vanishes on $H$.
Then
\begin{enumerate}[(i)]
\item \label{item:i} $H$ is Levi-flat,
\item \label{item:ii} the topological closure $\overline{H^*}$ is a union of complex
hyperplanes\footnote{Note that $\overline{H^*}$ may be a proper subset of
$H$ even if $H$ is irreducible.  See Example~\ref{example:whitney}.},
\item $\dim H_s = 2n-2$ or 
$\dim H_s = 2n-4$.
\end{enumerate}
If $\dim H_s = 2n-4$, then
there exist homogeneous coordinates $[z_0,\ldots,z_n]$ such that
\begin{enumerate}[(i),resume]
\item the set $H_s$ is given by $z_1 = z_2 = 0$,
\item $H$ is defined by a real bihomogeneous polynomial equation
$\rho(z_1,z_2,\bar{z}_1,\bar{z}_2) = 0$,
\item the holomorphic foliation given by $z_2 dz_1 - z_1 dz_2$ extends the
Levi-foliation of $H^*$.
\end{enumerate}
\end{thm}

A smooth hypersurface is Levi-flat if and only if it contains a complex
hypersurface through every point.  Therefore, \ref{item:i} follows directly from
\ref{item:ii}.  The germ of the complex hypersurface through any point is
unique, and hence \ref{item:ii} says that the only complex hypersurfaces
in $H^*$ are complex hyperplanes.

A real-analytic subvariety that has a global defining real-analytic
function is called $\C$-analytic.
Equivalently such a subvariety is the real trace of a complex analytic subvariety of a
neighbourhood in the complexification.
See \cite{GMT:topics} for more information.
For example, an algebraic real-analytic subvariety of $\bP^n$ is
$\C$-analytic: Suppose $H \subset \bP^n$ is given in homogeneous
coordinates as the zero set of a
degree $(k,k)$ bihomogeneous polynomial $r(z,\bar{z})$, that is,
$r(\lambda z,\overline{\lambda z}) = \abs{\lambda}^{2k} r(z,\bar{z})$.
The expression
$\frac{r(z,\bar{z})}{\snorm{z}^{2k}}$ defines a real-analytic function
on all of $\bP^n$, whose zero set is $H$.
The theorem therefore holds for all real algebraic Levi-flat hypersurfaces,
though there exist many non-algebraic $\C$-analytic subvarieties in $\bP^n$ as
well.
With the techniques used in the proof, it is impossible to avoid the condition
that $H$ is contained in a $\C$-analytic set. 
It is, however, a natural requirement
as a non-$\C$-analytic subvariety may be
several unrelated subvarieties glued together, see e.g.\ \cite{Lebl:noteex}.

The function $\frac{z_1}{z_2}$ is the rational first integral of
the foliation given by $z_2 dz_1 - z_1 dz_2$.  In case $\dim H_s = 2n-4$,
in the $(z_1,z_2)$-space, $H$ is a complex cone over a nonsingular algebraic curve.
Note that homogeneous coordinates on $\bP^n$ are given up to
an automorphism of $\bP^n$, i.e.\ an invertible linear map of
$\C^{n+1}$.

If $H$ is induced by a curve in the Grassmanian,
we say 
a point $p \in H_s$ is \emph{degenerate} if infinitely many complex hyperplanes
in $H$ pass through $p$.  
If $\dim H_s = 2n-4$, then every point in $H_s$ is degenerate.
Our next result says that
if $H$ has a degenerate singularity then $H$ arises from an algebraic
surface in lower dimension.  A curve in $\bP^1$ has no
degenerate singularities.

\begin{thm} \label{thm:thm2}
Suppose $H \subset \bP^n$ is a real-analytic subvariety of
real codimension one
induced by an irreducible real-analytic curve in $G(n+1,n)$, further suppose $H_s$
contains a degenerate singular point.  Then
there exists an integer $1 \leq \ell < n$, homogeneous coordinates
$[z_0,\ldots,z_n]$, and a bihomogeneous real polynomial
$\rho(z_0,\ldots,z_\ell,\bar{z}_0,\ldots,\bar{z}_\ell)$, such that
\begin{enumerate}[(i)]
\item $H$ is the zero set of $\rho$ (in particular, $H$ is algebraic),
\item $\rho$ defines a Levi-flat algebraic hypersurface $X \subset \bP^\ell$
with no degenerate singularities
induced by a curve in $G(\ell+1,\ell)$.
\end{enumerate}
\end{thm}

Theorem~\ref{thm:thm2} implies that a nonalgebraic $H$ has no degenerate singularities.
The assumption that the curve in $G(n+1,n)$ is irreducible is equivalent 
to assuming $H$ is induced by a nonsingular germ of a curve in $G(n+1,n)$.
The theorem also justifies our claim that a subvariety in $G(n+1,n)$ does
not necessarily induce a subvariety in $\bP^n$.  If we take a nonalgebraic
curve in $G(n+1,n)$ where the planes all go through a single point, then
there cannot be a proper subvariety of $\bP^n$ containing all the hyperplanes as
Proposition~\ref{prop:Csubvar} would imply it be nonalgebraic, and that
contradicts the theorem.

\begin{example}[Whitney umbrella with complex lines]\label{example:whitney}
Let us discuss a crucial algebraic example in $\C^2$ (and hence in $\bP^2$)
showing 
that $H$ itself may not be a union of complex
hyperplanes, and that the Levi-foliation need not extend.
Let $(z,w) \in \C^2$ be our coordinates and write $z=x+iy$ and $w=s+it$.
Consider the hypersurface $H$ given by
\begin{equation}
sy^2-txy-t^2 = 0 .
\end{equation}
The hypersurface is Levi-flat with the nonsingular
hypersurface part foliated by the complex lines given by $w=c z + c^2$,
where $c \in \R$.  Indeed, setting $z=\xi$ and $w=c \xi + c^2$ gives a
$\overline{H^*}$ as an image of $\C \times \R$.
This irreducible hypersurface
has a lower-dimensional totally-real component given by $\{ t=0 , y = 0 \}$,
``half of it'' (for $s$ sufficiently negative) sticking out of the hypersurface
as the ``umbrella handle.''  In the complement of $\{ y = 0 \}$, the
hypersurface $H$ is a graph:
$s = \frac{txy+t^2}{y^2}$ or $s = x \bigl( \frac{t}{y} \bigr) +
{\bigl(\frac{t}{y}\bigr)}^2$.  The set
$\{ t=0 , y=0 , s < \frac{-x^2}{4} \}$ is not in the closure of
$H^*$.
Therefore, not all of $H$ is a union of
complex lines; only the closure $\overline{H^*}$ is such a of complex lines, as claimed in 
Theorem~\ref{thm:thm1}.  
Notice that when $x=0$, we obtain the standard Whitney
umbrella in $\R^3$ \cite{Whitney:u}.

\begin{figure}[ht]
\includegraphics{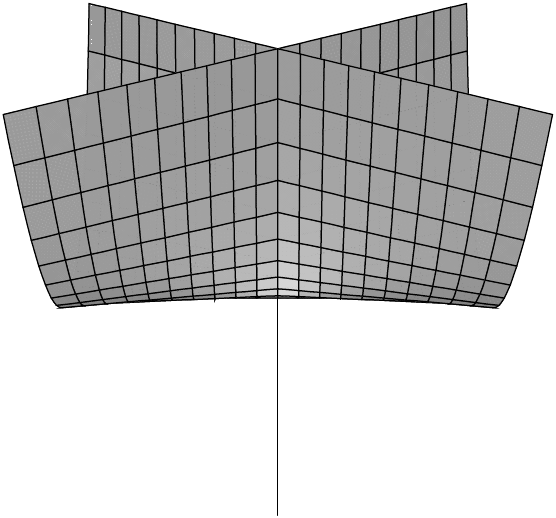}
\caption{The Whitney Umbrella in $\R^3$.\label{fig:whitney}}
\end{figure}

The singular set $H_s = \{ t=0 , y = 0, s \geq \frac{-x^2}{4} \}$ is of
real dimension 2, and it is a generic manifold at points 
arbitrarily near the origin.  Therefore,
the Levi-foliation
does not extend as a singular holomorphic foliation of a neighbourhood of the origin.
The Levi-foliation does extend as a holomorphic $2$-web given by
${dw}^2+z\, dz\, dw-w\,{dz}^2=0$.
Its multivalued
first integral is obtained by applying the quadratic formula to
$w=cz + c^2$.  See \cites{Brunella:lf,ShafikovSukhov}.
\end{example}

See \cite{Lebl:projlf} for an example of an umbrella-type subvariety
composed of complex lines where the ``umbrella handle'' is a complex line.

\begin{example}
A simple but illustrative
example is the hypersurface $H$ in $\bP^2$ given in
homogeneous coordinates $[z_0,z_1,z_2]$ as
\begin{equation}
\abs{z_1}^2=\abs{z_2}^2 .
\end{equation}
The singularity is the set $\{ z_1 = z_2 = 0 \}$; that is, a single point in
$\bP^2$, which is a degenerate singularity.  The singular holomorphic foliation extending the Levi-foliation
is given by the one-form $z_2 dz_1 - z_1 dz_2$ as claimed, and
$H$ is the cone over a circle in $\bP^1$.
\end{example}

One of the primary motivations of this paper is to find a nonalgebraic
singular Levi-flat hypersurface in projective space.
%
%
Previously in \cite{Lebl:projlf},
the author proved an analogue of Chow's theorem for Levi-flat hypersurfaces
under extra hypotheses:  A Levi-flat subvariety whose
Levi-foliation extends as a singular holomorphic foliation at each point and
has infinitely many compact leaves is contained in a real-algebraic
subvariety.  It should be noted that a real-algebraic Levi-flat
automatically has compact leaves.


\begin{thm} \label{thm:nonalgexample}
There exists a real-analytic singular Levi-flat hypersurface $H \subset \bP^2$
with all leaves complex hyperplanes,
such that $H^*$ (and hence $H$)
is not contained in any proper real-algebraic subvariety
of $\bP^2$.
\end{thm}

Therefore,
Chow's theorem does not hold for Levi-flat hypersurfaces without extra
assumptions.  Also, the hypotheses
in~\cite{Lebl:projlf} are not unnecessary;  the
nonalgebraic
hypersurface we construct has compact leaves, but the Levi-foliation does
not extend at certain points of its singular set.

We further construct an example of a real-analytic singular Levi-flat hypersurface
$H$ that is semialgebraic, but not algebraic.  That is, we construct an
irreducible algebraic Levi-flat hypersurface globally reducible into
two components as a real-analytic subvariety.

The author would like to acknowledge Xianghong Gong for suggesting
the original question.  The author would like to
thank Arturo Fern{\'a}ndez-P{\'e}rez for many 
discussions on the topic, and
several very useful suggestions and corrections to the manuscript.

\section{Real-analytic subvarieties induced by curves in the Grassmannian}

A real-analytic subvariety $H \subset \bP^n$
is \emph{algebraic} if it is the
zero set of a bihomogeneous polynomial.  
A set is \emph{semialgebraic}, if it is defined by
a finite set of algebraic equalities and inequalities.
A real-analytic subvariety can be semialgebraic, but not algebraic;
a real-analytic algebraic subvariety can be irreducible as an algebraic
subvariety, but reducible as a real-analytic subvariety.

Let us fix some notation.
Let $\sigma \colon \C^{n+1} \setminus\{0\} \to \bP^n$ be the natural projection.
If $X$ is a real-analytic subvariety of $\bP^n$, then let
$\tau(X) \subset \C^{n+1}$ to be the set of points $z \in \C^{n+1}$ such that
$\sigma(z) \in X$ or $z = 0$.
We say $Y \subset \C^{n+1}$ is a \emph{complex cone} if $z \in Y$ implies
$\lambda z \in Y$ for all $\lambda \in \C$.  For example, $\tau(X)$ is a
complex cone.  If $Y \subset \C^{n+1}$ is a complex cone we
write $\sigma(Y)$ for the induced set in $\bP^n$.

While $X \subset \bP^n$ may be a real-analytic subvariety, the set $\tau(X)$
need not be a subvariety.  The problem is at the origin, clearly $\tau(X)
\setminus \{ 0 \}$ is a subvariety of $\C^{n+1} \setminus \{ 0 \}$.
This is precisely where the proof of Chow's
theorem breaks down in the general real-analytic case.  Via
the standard proof of Chow's theorem (see e.g.\ \cite{Lebl:projlf}), if
$\tau(X)$ is a variety, then $X$ is algebraic.  If $\tau(X)$ is semianalytic
(in particular contained in a real-analytic subvariety of the same dimension), then 
$X$ is semialgebraic.

Before we do anything else, we prove the following proposition promised in
the introduction.  It means that every subvariety $H \subset \bP^n$ induces a
subvariety (possibly empty) in $G(n+1,n)$ of hyperplanes contained in $H$.  Therefore,
from now on we can assume the curve that induces our hypersurfaces is
a closed real-analytic subvariety of dimension 1.

\begin{prop} \label{prop:Csubvar}
Suppose $X \subset \bP^n$ is a proper real-analytic subvariety.  Then the
set $C \subset G(n+1,n)$ of complex hyperplanes contained in $X$ is a
real-analytic subvariety of $G(n+1,n)$ of dimension 1 or less.
Furthermore, if $X$ is algebraic, then $C$ is algebraic.
\end{prop}

\begin{proof}
Suppose $L$ is a hyperplane in $G(n+1,n)$ passing through a point $p \in
\bP^n$.
Without loss of generality, let $p$ be the origin in inhomogeneoous
coordinates given by $z_0 = 1$,
and suppose $L$ is given by $z_1 = 0$.
The hyperplanes in $G(n+1,n)$ in a neighbourhood of $L$ are given 
by the equation
\begin{equation}
z_1 = a_0 + a_2 z_2 + \cdots + a_n z_n
\end{equation}
for $a_j$ small.  The corresponding hyperplane is contained in $X$
if as a function of $z_2,\ldots,z_n$ and $\bar{z}_2,\ldots,\bar{z}_n$,
\begin{equation}
\rho( a_0 + a_2 z_2 + \cdots + a_n z_n, z_2, \ldots, z_n ,
\overline{a_0 + a_2 z_2 + \cdots + a_n z_n}, \bar{z}_2, \ldots, \bar{z}_n ) \equiv 0
\end{equation}
for all real-analytic $\rho$ that vanish on $X$ near $p$.  Expanding as a series in
$z_2,\ldots,z_n,\bar{z}_2,\ldots,\bar{z}_n$
\begin{equation}
\sum
c_{j_2,\ldots,j_n,k_2,\ldots,k_n}(a_0,a_2,\ldots,a_n,\bar{a}_0,\bar{a}_2,\ldots,\bar{a}_n)
z_2^{j_2} \cdots z_n^{j_n}
\bar{z}_2^{k_2} \cdots \bar{z}_n^{k_n},
\end{equation}
and setting its coefficients to 0, we obtain real-analytic equations
for $a_0$, $a_2$,\ldots, $a_n$.  The result follows.

The dimension claim follows easily; take a two-dimensional submanifold 
of $C(n+1,n)$ and note that the corresponding hyperplanes fill
an open subset of $\bP^n$.
Finally if $X$ is algebraic then take only polynomial $\rho$ above and
the equations giving $C$ are polynomial and hence $C$ is algebraic.
\end{proof}

The proposition has the following corollary.  If $H$ is induced by
an irreducible curve $C \subset G(n+1,n)$, then $H$ is irreducible.  If $H$
was reducible into two smaller components, at least one of them
contains all the complex hyperplanes from $C$.

We have the following special case of a well-known and easy observation.

\begin{prop}
Suppose $H \subset \bP^n$ is a real-analytic subvariety 
of real codimension one 
induced by a real-analytic curve in $G(n+1,n)$.  Then
$2n-4 \leq \dim H_s \leq 2n-2$.
\end{prop}

\begin{proof}
A smooth hypersurface cannot contain two distinct complex hyperplanes through a
point.  Thus any point where two hyperplanes meet must be singular.
Any two complex hyperplanes in $\bP^n$ meet along a set
of complex dimension $n-2$ or real dimension $2n-4$.
\end{proof}


Before getting into the details of the proof of Theorem~\ref{thm:thm1}
let us start with complexification of $\bP^n$ as a real-analytic manifold.
We write $\sM$ for the obvious
complexification of $\bP^n$; let $z = (z_0,\ldots,z_n)$ be the homogeneous
coordinates, then we let $w = \bar{z}$ be a new variable
and we have $\sM \cong \bP^n \times \bP^n$.
Then $\bP^n$ is imbedded as a real-analytic totally-real submanifold
via $z \mapsto (z,\bar{z})$.
It is a classic result of Grauert (see also~\cite{GMT:topics}) that 
a real-analytic (paracompact)
manifold has a fundamental system of Stein neighbourhoods
in its complexification.
This result is not completely necessary in the proof below, but it
simplifies the proof.
This fact may seem a little counterintuitive
about $\bP^n$, but recall we are talking about $\bP^n$ only as
a real-analytic submanifold, ignoring the complex structure.

The first two claims of Theorem~\ref{thm:thm1} follow from the following
lemma.

\begin{lemma} \label{lemma:22}
Suppose $H \subset \bP^n$ is a real-analytic subvariety 
of real codimension 1,
induced by an irreducible real-analytic curve in $G(n+1,n)$.
Suppose there
exists a nontrivial real-analytic function $r$ defined in a neighbourhood
$V$ of $H$, such that $r$ vanishes on $H$.
Then
$\overline{H^*}$ is a union of complex hyperplanes.
\end{lemma}

\begin{proof}
Take a small smooth piece of the curve.
In homogeneous coordinates on $\bP^n$, we write the
corresponding planes 
(after perhaps reordering coordinates) as
\begin{equation}
z_0 = a_1(t) z_1 + \cdots + a_n(t) z_n .
\end{equation}
where $a_j(0) = 0$ for all $j$, but $a_k'(0) \not= 0$ for at least one $k$.
By implicit function theorem, for most points on the leaf $z_0 = 0$, we can
solve for $t$, and hence near such points the union of the planes is a $2n-1$ dimensional
real submanifold.
In particular, there is a point $q \in H^*$, a neighbourhood of which in $H^*$ is locally a union
of complex hypersurfaces.

We complexify as above and consider $\bP^n \subset \sM$ as a
totally-real subset.
Let $r$ and $V$ be as given.
This real-analytic function complexifies to
some neighbourhood $U$ of $H$ in the complexification $\sM$.
As $V$ is a real-analytic submanifold it has a fundamental system of Stein
neighbourhoods in $\sM$ and we can therefore assume that $U$ is Stein.
The zero set
of $r$ is a complex analytic subvariety $\sX$ of $U$ such that
$H \subset \sX \cap \bP^n$.
Assume $\sX$ is irreducible (as $H$ is irreducible).
As $\sX$ is complex analytic in a Stein domain, it is coherent.
We can assume 
the complexified $r$ to be the defining function for $\sX$
such that for every $p \in \sX$, the function $r$ generates the ideal
$I\bigl((\sX,p)\bigr)$
of germs at $p$ of holomorphic functions vanishing on the germ $(\sX,p)$.
In particular, the derivative $dr$ vanishes only on the singular locus of $\sX$.
The singular locus of $\sX$ intersected with $H$ is a proper subvariety of
$H$ of dimension less than $2n-1$.  Hence $dr$ vanishes on a subvariety of
$H$ of dimension less than $2n-1$.  

Let $(H^*)' = \{ z \in H^* : dr(z,\bar{z}) \not= 0 \}$.  The set $(H^*)'$ is
dense and open in $H^*$.
Therefore we assume that $q \in (H^*)'$.
The real form $\theta = i(\partial r - \bar{\partial} r)$ defines the space
$T_q^{(1,0)} H \oplus T_q^{(0,1)} H$.  This form complexifies and also
defines a (complex) codimension 1 subspace $\sT$
on the tangent space of $\sX$ at smooth points, which agrees with
$T_p^{(1,0)} H \oplus T_p^{(0,1)} H$ at all $p \in (H^*)'$.  As a form on
$H^*$, $\theta$ is integrable near $q$ (as $H$ is Levi-flat there),
then the complexified $\theta$ is integrable on the regular points of $\sX$
and hence on all of $(H^*)'$.

So $H^*$ is Levi-flat, we now need
to show that the leaves are locally complex hyperplanes.
It is again enough to do so on a dense subset of $H^*$.  To make things
simpler let us move to inhomogeneous coordinates with $z_0 = 1$, and note
that $H^* \cap \{ z_0 = 1 \}$ is dense in $H^*$.
Compute the Hessian (the full real Hessian) $Hr$ of $r$, and note
that it vanishes on
$T_q^{(1,0)} H \oplus T_q^{(0,1)} H$ for the point $q \in (H^*)'$ where
$H$ is a union of complex hyperplanes.
The function $Hr$ also complexifies (the
complexification is the Hessian of the complexified $r$).
The complexified $Hr$ is identically zero on the vectors in $\sT$ at all
points in $\sX$ near $(q,\bar{q})$.
By identity $Hr$ vanishes on all the vectors in $\sT$ at all the
regular points of $\sX$ for which neither $z_0 = 0$ nor $w_0 = 0$
(the set of regular points of $\sX$ is connected).

Then at all points $p \in (H^*)' \setminus \{ z_0 = 0 \}$,
the Hessian $Hr$ also vanishes on
$T_p^{(1,0)} H \oplus T_p^{(0,1)} H$.  The leaves of the Levi-foliation near
$p$ must then be complex hyperplanes (no curvature).
That means the leaves near all points of $H^*$ are complex hyperplanes.

Take a complex hyperplane $L$ such that an open neighbourhood of it
is in $H^*$.  Clearly $L \subset H$.
What is left to show is that $L$ is in the closure of $H^*$.
Suppose we pick a point in $H^*$ near which the leaves are complex hyperplanes.
At such a point, taking a nonsingular curve $\gamma$ in $H^*$ transverse
to the leaves, we obtain that the leaves nearby clearly form a small nonsingular curve
in $G(n+1,n)$.
By the argument in the beginning of the proof, on a open dense subset
of $L$, the set $L$ is part of a smooth $2n-1$ dimensional manifold.  In
other words, the leaf is in the closure $\overline{H^*}$.

The limit set of complex planes is a union of complex planes, and hence
$\overline{H^*}$ is a union of complex planes.
\end{proof}

\section{Degenerate singularities}

To prove 
Theorem~\ref{thm:thm2}, it is enough to prove the following lemma.
The lemma will also be useful in the proof of Theorem~\ref{thm:thm1}.

\begin{lemma} \label{lemma:alg}
Suppose $H \subset \bP^n$ is a real-analytic subvariety of
real codimension one
induced by an irreducible real-analytic curve in $G(n+1,n)$ and suppose
$H_s$ has a degenerate singularity.  Then there 
exist homogeneous coordinates $[z_0,\ldots,z_n]$ such that
$H$ has a real polynomial defining function depending on $z_0$
through $z_{n-1}$ only.
\end{lemma}

To obtain Theorem~\ref{thm:thm2}, that is, to obtain an $X$ with no nondegenerate
singularities, we repeatedly apply the lemma, noting 
a curve in $\bP^1$ cannot have
degenerate singularities.
Then for the $\ell$ in the theorem
there exists an $X \subset \bP^\ell$ defined by the 
polynomial in $z_0$ through $z_\ell$,
such that $\tau(H) = \tau(X) \times \C^{n-\ell}$.
To prove the lemma, it is enough to show that
in inhomogeneous coordinates where the origin is the degenerate singularity,
$H$ is a complex cone defined by a bihomogeneous polynomial,
since then we have eliminated one variable from the defining equation
and we have shown algebraicity.

\begin{proof}
Let us work in inhomogeneous coordinates defined by $z_0 = 1$, and suppose
the degenerate point is the origin.
We proceed as in the standard proof of Chow's theorem.
Let $\rho$ be a real-analytic defining function for $H$ near the origin.
Decompose $\rho$
into bihomogeneous components as
\begin{equation}
\rho(z,\bar{z}) = \sum_{jk} \rho_{jk}(z,\bar{z}).
\end{equation}
We have infinitely many complex hyperplanes passing through the origin.


Let $V$ be the union of the hyperplanes through the origin contained in
$H$ near the origin.  If $z$ in $V$, then $\lambda z \in V$
for all $\lambda  \in \C$, and so
\begin{equation}
0 = \rho(\lambda z,\bar{\lambda }\bar{z}) = \sum_{jk} \lambda^j
\bar{\lambda}^k \rho_{jk}(z,\bar{z}) .
\end{equation}
So for each $j$ and $k$, $\rho_{jk}(z,\bar{z}) = 0$ for $z \in V$.
Let $\widetilde{H}$ be the algebraic variety defined by
$\rho_{jk}(z,\bar{z}) = 0$ for all $j$ and $k$.  Clearly as germs at the origin
$(\widetilde{H},0) \subset (H,0)$.  Furthermore $\widetilde{H}$ must be of real dimension $2n-1$,
and $\widetilde{H}$ is complex cone as it is defined by bihomogeneous
polynomials.

If $z \in \widetilde{H}$, then $\lambda z \in \widetilde{H}$ for all
$\lambda \in \C$.  For small $\lambda$, $\lambda z \in H$ as 
$\widetilde{H} \subset H$ as germs at the origin.  But as $H$ is a
subvariety we have that $\lambda z \in H$ for all $\lambda \in \C$, in
particular if $\lambda = 1$, and so $z \in H$.
We have shown that $\widetilde{H} \subset H$.  As the curve of complex
hyperplanes in $H$ is irreducible and $H$ is the smallest subvariety
contining those planes $\widetilde{H} = H$.
Hence $H$ is algebraic.

Finally, as $H$ is a complex cone in $\C^n$
it induces a subvariety $X \subset \bP^{n-1}$.  In the right coordinates
$\tau(H) = \tau(X) \times \C$.  The lemma follows.
\end{proof}

We remark the lemma implies that to study hypersurfaces induced by
curves in the Grassmannian, it is enough to study those that have no
degenerate singularities.

\section{Extending the foliation when the singularity is small}

Given a subvariety $C \subset G(n+1,n)$, let $S \subset \bP^n$ be the set of points
where two or more complex hyperplanes in $C$ intersect.  The set $S$ is
given in homogeneous coordinates by $\{ z \in \C^{n+1} : z \cdot a = z \cdot
b = 0 , [a]
\not= [b], [a] \in C, [b] \in C \}$.  That is a subanalytic set, see e.g.\ \cite{BM:semisub}.

\begin{lemma} \label{lemma:findp}
Let $C \subset G(n+1,n)$, $n \geq 2$, be an irreducible real-analytic
subvariety of real dimension one.
Let $S \subset \bP^n$ be the set of points
where two or more complex hyperplanes in $C$ intersect.  If $\dim_{\R} S
\leq 2n-3$, then after an automorphism of $\bP^n$, every hyperplane in $C$
is given by an equation of the form
\begin{equation}
0 = a_1 z_1 + a_2 z_2,
\end{equation}
and the set $S$ is given by the equations $z_1 = z_2 = 0$.  In particular,
$\dim_{\R} S = 2n-4$.
\end{lemma}

\begin{proof}
Let $z_0,z_1,\ldots,z_n$ be the homogeneous coordinates on $\bP^n$.  An
automorphism of $\bP^n$ is an invertible linear transformation in $n+1$
variables.
Pick two distinct smooth points on $C$.
Work in inhomogeneous coordinates given by $z_0 = 1$.
Without loss of generality, after a change of
coordinates, assume the two
hyperplanes are given by $z_1 = 0$ and $z_2 = 0$.  We parametrize
the curve $C$ near the two points and hence the corresponding 
hyperplanes as
\begin{align}
& z_1 = a_0(s) + a_2(s) z_2 + a_3(s) z_3 + \cdots + a_n(s) z_n ,
\\
& z_2 = b_0(t) + b_1(t) z_1 + b_3(t) z_3 + \cdots + b_n(t) z_n .
\end{align}
The functions $a_k$ and $b_k$ are real-analytic functions of one
real variable defined near the origin and vanishing at $0$.

The hypotheses of the lemma are that the intersection of these two
families of planes is of dimension less than or equal to $2n-3$.  Furthermore we can
assume that as a vector valued function, the $a$ and the $b$ is
non-constant.  That is, not all $a_k$ are identically zero, and similarly not
all $b_k$ are identically zero.  We will show that 
$a_0$ and $a_3,\ldots,a_n$ must be identically zero, and 
$b_0$ and $b_3,\ldots,b_n$ must be identically zero.  That is, only $a_2$
and $b_1$ are allowed to vary (they must vary in fact).  Once this is
proved, the lemma holds.

Let us suppose for contradiction the above assertion is not true.
In particular,
after a linear change of variables in $z_0$, $z_3,\ldots,z_n$, assume
$b_0$ and $a_0$ are not both identically zero.  By symmetry,
assume it is $a_0$ that is not identically zero.  If $b_0$ is
identically zero after a generic linear change of variables as above,
then $b_3$ through $b_n$ are identically
zero as well and in that case $b_1$ is not identically zero.

Solving for $z_1$ and $z_2$, we obtain a piece of $S$ parametrized by $s$,
$t$, and $z_3,\ldots,z_n$ as
\begin{align}
& z_1 = \frac{a_0(s) + a_2(s)
\bigl(b_0(t) + b_3(t) z_3 + \cdots + b_n(t) z_n\bigr) + a_3(s) z_3 + \cdots + a_n(s) z_n}{1-a_2(s)b_1(t)} ,
\\
& z_2 = \frac{b_0(t) + b_1(t) \bigl(a_0(s) + a_3(s) z_3 + \cdots +
a_n(s) z_n \bigr)  + b_3(t) z_3 + \cdots + b_n(t) z_n}{1-a_2(s)b_1(t)} .
\end{align}
where $z_3,\ldots,z_n$ simply parametrize themselves.  As dimension of
the image is less than or equal to $2n-3$, it must be true when we set $z_3$ through $z_n$
to zero, the rank of the resulting mapping (rank of the derivative)
\begin{align}
& z_1 = \frac{a_0(s) + a_2(s) b_0(t)}{1-a_2(s)b_1(t)} ,
\\
& z_2 = \frac{b_0(t) + b_1(t) a_0(s)}{1-a_2(s)b_1(t)} ,
\end{align}
must be of rank strictly less than $2$ at all points.  Clearly this is possible if both
$a_0$ and $b_0$ are both identically 0, however, we supposed for
contradiction $a_0$ is not identically zero.

If neither $a_0$ nor $b_0$ is identically zero, then the
mapping is a finite map.  That is, solving
$a_0(s) + a_2(s) b_0(t) = 0$, and 
$b_0(t) + b_1(t) a_0(s) = 0$ for $s$ and $t$ near zero implies that both
$a_0$ and $b_0$ must be zero.  A finite map has generically full rank, that
is 2, which contradicts our assumption on rank.

The final case to check is when $b_0$ is identically zero.
The mapping becomes
$\frac{a_0(s)}{1-a_2(s)b_1(t)}$ and
$\frac{b_1(t) a_0(s)}{1-a_2(s)b_1(t)}$, and that is easily seen to have
derivative of rank 2 for generic points $s,t$ near the origin if
$b_1$ is not identically zero.  By assumption, $b_1$ is not identically zero.

By irreducibility of $C$ we find that if $S$ is contained in all complex hyperplanes
in some open set in $C$, then $S$ is contained in all complex hyperplanes
in $C$.
\end{proof}

The following lemma proves the remainder of the claims in
Theorem~\ref{thm:thm1}.

\begin{lemma} \label{lemma:32}
Suppose $H \subset \bP^n$ is a real-analytic subvariety of
real codimension one
induced by an irreducible real-analytic curve in $G(n+1,n)$ and suppose 
$\dim H_s \leq 2n-3$.  Then
after an automorphism of $\bP^n$,
there exist homogeneous coordinates $[z_0,\ldots,z_n]$ such that
\begin{enumerate}[(i)]
\item $H$ has a real polynomial defining function depending on $z_1$ and $z_2$ only.
\item The holomorphic foliation given by $z_2 dz_1 - z_1 dz_2$ extends the
Levi-foliation of $H^*$.
\item The set $H_s$ is given by $z_1 = z_2 = 0$ and so in particular
$\dim_{\R} H_s = 2n-4$.
\end{enumerate}
\end{lemma}

\begin{proof}
Let $C \subset G(n+1,n)$ be the subvariety defined by all complex
hyperplanes in $H$.  Let $C_1$ be an irreducible one-dimensional
component of $C$.
Use Lemma~\ref{lemma:findp} to find the coordinates such that
all hyperplanes in $C_1$ are given by linear functions of $z_1$ and $z_2$.
In particular all the hyperplanes pass through the origin in the
inhomogeneous coordinates given by $z_0 = 1$.

The origin is a degenerate singularity and therefore by Lemma~\ref{lemma:alg},
$H$ is algebraic and defined by a bihomogeneous polynomial.  


The foliation is given by the one-form $z_2 dz_1 - z_1 dz_2$ with the
meromorphic first integral $\frac{z_1}{z_2}$, as the leaves of this
foliation are precisely the complex hyperplanes given by linear functions of
$z_1$ and $z_2$ only.  This foliation then clearly extends the Levi-foliation
of $H$.

Let us now work in homogeneous coordinates $z_0,\ldots,z_n$.  The cone
$\tau(H)$ is an algebraic subvariety in $\C^{n+1}$.  By fixing $z_0$, and
$z_3,\ldots,z_n$ at a generic value we obtain a real polynomial in $z_1$ and $z_2$ that
vanishes on a set of real dimension 3 in the $(z_1,z_2)$-space.  Since it
must vanish on the hyperplanes in $H$ it vanishes on $H$.  We obtain
our defining polynomial depending on $z_1$ and $z_2$ only.
Hence $H$ is a complex cone over a curve in $\bP^1$.

The singular set of $H$ is therefore either given by the singular set of
the foliation, which is precisely $z_1 = z_2 = 0$, or by the complex hyperplanes
that correspond to the singular points of
the curve in $\bP^1$.  That is, $H_s$ is a union of the set
$\{ z_1 = z_2 = 0 \}$ and a finite number of complex hyperplanes
corresponding to the singular points of the curve.  That is $H_s$ is
a complex subvariety and hence by assumption on dimension
it must be of real dimension $2n-4$
and cannot contain any complex hyperplanes.
So the curve in $\bP^1$
must be nonsingular and $H_s = \{ z_1 = z_2 = 0 \}$.
\end{proof}

\section{The construction of a nonalgebraic hypersurface}

In this section we construct
a real-analytic singular Levi-flat hypersurface $H \subset \bP^2$
with all leaves compact, such that $H^*$
is not contained in any proper real-algebraic subvariety
of $\bP^2$.  That is, the only bihomogeneous polynomial vanishing on
all of $\tau(H^*)$ and hence on $\tau(H)$ is identically zero.  The $H$ we construct will be induced
by a nonalgebraic real-analytic curve in $G(3,2)$.

\begin{lemma} \label{lemma:construct}
Let $X \subset \R^2$ be a connected compact real-analytic curve
with no singularities.
Let $\widetilde{H}$ be the complex cone defined by
\begin{equation}
\widetilde{H} = \{ (z_0,z_1,z_2) \in \C^3 : z_0 = z_1 x + z_2 y \text{ where $(x,y) \in X$}
\} \cup \{ z \in \C^3 : z_1\bar{z}_2 = \bar{z}_1z_2 \} .
\end{equation}
Then $\widetilde{H}$ is a subvariety in $\C^{3} \setminus \{ 0 \}$.  In
other words the induced set $\sigma(\widetilde{H})$ is a real-analytic
subvariety of $\bP^2$.
\end{lemma}

\begin{proof}
Let $X \subset \R^2$ be as above.
Suppose $a,b \colon \R \to X$ are two real-valued real-analytic functions of
a real variable $t$ such that the image of $\bigl(a(t),b(t)\bigr)$ is all of
$X$.
As $X$ has no singularities, 
we assume the derivatives $a'(t)$ and $b'(t)$ do not both vanish
at any $t \in \R$.

Define a complex cone $\widetilde{H}$ as
\begin{equation}
\widetilde{H} = \{ (z_0,z_1,z_2) \in \C^3 : z_0 = z_1 x + z_2 y \text{ where $(x,y) \in X$}
\} \cup \{ z \in \C^3 : z_1\bar{z}_2 = \bar{z}_1z_2 \} .
\end{equation}
Clearly $\widetilde{H}$ is a complex cone and so $H = \sigma(\widetilde{H})$ is a
well-defined subset of $\bP^2$.  If $\widetilde{H} \setminus \{ 0 \}$ is a real-analytic
subvariety of $\C^3 \setminus \{ 0 \}$, then $H$ is a real-analytic
subvariety of $\bP^2$.  Assuming $\widetilde{H}$ is a real-analytic
subvariety away from the origin then $H$ is Levi-flat,
as $\widetilde{H}$ is a union of complex
hyperplanes.  In fact, $H$ is induced by a
nonsingular real-analytic curve in $G(3,2)$.

It is therefore left to show $\widetilde{H}$ is a subvariety away from
the origin.
We write
\begin{equation}
\begin{bmatrix}
z_0 \\ \bar{z}_0
\end{bmatrix}
=
\begin{bmatrix}
z_1 & z_2 \\ \bar{z}_1 & \bar{z}_2
\end{bmatrix}
\begin{bmatrix}
x \\ y
\end{bmatrix} .
\end{equation}
If the matrix is invertible we solve for $x$ and $y$ in terms of $z$.
Outside the set
\begin{equation}
\Delta = \{ z : z_1\bar{z}_2 - \bar{z}_1z_2  = 0 \} ,
\end{equation}
the set $\widetilde{H}$ is an image of a
real-analytic submanifold under a real-analytic diffeomorphism, and
therefore a subvariety.

Let us take a point on $\widetilde{H}$ where
$z_1\bar{z}_2 - \bar{z}_1z_2  = 0$.  Call this point $p = (z_0^0,z_1^0,z_2^0)$.
The complex values $\bar{z}_1^0$ and $\bar{z}_2^0$ lie on the same line
through the origin in $\C$.  After moving all of $\C^3$ by a diagonal unitary
we assume $z_1^0$ and $z_2^0$ are real.  As they are not both zero, we
find a $2 \times 2$ real orthogonal matrix $U$
that takes the point $(1,0)$ to $(z_1^0,z_2^0)$.
Write $A(t) = \bigl(a(t),b(t)\bigr)$.  Then near $p$ and outside of
$\Delta$, the set
$\widetilde{H}$ is the image of
\begin{equation}
(\xi_1,\xi_2,t) \mapsto
\bigl((U\xi \cdot A(t)), U\xi \bigr)
=
\bigl((\xi \cdot U^tA(t)), U\xi \bigr) ,
\end{equation}
where $\xi_1,\xi_2$ are complex and $t$ real.
Again by moving around (post-composing the above mapping) by an orthogonal
mapping we assume 
$z_1^0 = 1$ and 
$z_2^0 = 0$.

We now work in inhomogeneous coordinates in $\C^2$
where we set $z_1=1$ and write $z_2 = \zeta$.
The set $\Delta$ is the set where $\zeta$ is
real-valued.  It is
enough to consider the image of the map
\begin{equation}
(\zeta,t) \mapsto \bigl(a(t)+\zeta b(t), \zeta \bigr)
\end{equation}
for $\zeta \in \C$ near the origin and $t \in \R$.  The question
is local and 
so it is enough to consider $t$ near the origin.  We must show
the image union $\Delta$ is a real-analytic subvariety.

We complexify by letting $\bar{\zeta} = \omega$
and $\bar{t} = s$, treating $t$ as a complex variable.
As neither $a$ nor $b$
vanish identically,
the complexified mapping (the coefficients of $a$
and $b$ are real)
\begin{equation}
(\zeta,t,\omega,s) \mapsto \bigl(
a(t)+\zeta b(t),
\zeta,
a(s)+\omega b(s),
\omega
\bigr)
\end{equation}
is a finite map near the origin.
Hence it is a proper holomorphic mapping between two
neighbourhoods of the origin.  By proper mapping theorem this map takes
the set $s=t$ to a complex analytic subvariety of a neighbourhood of the
origin.  We restrict to the ``diagonal,'' to what this image is intersected
with the original $\C^2$ before complexification:
\begin{equation}
\zeta = \overline{\omega}, \qquad
a(t)+\zeta b(t)
=
\overline{
a(s)+\omega b(s)
} .
\end{equation}
We substitute
$\zeta = \overline{\omega}$ and $s=t$ into the right hand side of the second
equation:
$\overline{
a(s)+\omega b(s)
}
=
a(\bar{t})+ z b(\bar{t})$.  So
\begin{equation}
\bigl( a(t) - a(\bar{t}) \bigr) + \zeta  \bigl( b(t) - b(\bar{t})\bigr) = 0
.
\end{equation}
As 
$a(t) - a(\bar{t}) = a(t) - \overline{a(t)}$ and $b(t) - b(\bar{t}) = b(t) -
\overline{b(t)}$ are purely imaginary for all $t \in \C$,
either 
$a(t) - \overline{a(t)}$ and
$b(t) - \overline{b(t)}$ are both zero, or $\zeta$ is real.
So $\widetilde{H}$ is contained in a real subvariety locally.
To show that $\widetilde{H}$ is a subvariety at $p$,
we need to show that the complexified image does not introduce any points
on ``the diagonal'' that are not in $\widetilde{H}$.

The derivatives of $a$ and $b$ do not both vanish, so 
without loss of generality $b'(0) \not= 0$.  Then 
$b(t) - \overline{b(t)} = 0$ implies $t$ is real.  Therefore
the image of the complexified mapping lies in $\C^2$ (on the diagonal), if
either $t$ is real or if $\zeta$ is real.  A point where
$t$ is real is in the image of
$\bigl(a(t) + \zeta b(t),\zeta\bigr)$ for a
real $t$ (and therefore in $\widetilde{H}$). A point where
$\zeta$ is real is in $\Delta$ and therefore in $\widetilde{H}$.  Either way
the image of $\bigl(a(t) + \zeta b(t),\zeta\bigr)$ for a real $t$ together
with $\Delta$ is a subvariety.  Therefore, $\widetilde{H}$ is a subvariety.
\end{proof}

We now use this construction to prove Theorem~\ref{thm:nonalgexample}.

\begin{proof}[Proof of Theorem~\ref{thm:nonalgexample}]
Let $X \subset \R^2$ be a connected compact real-analytic curve not contained
in any proper real-algebraic subvariety of $\R^2$.  That is, any bihomogeneous
polynomial $p(x,y)$ vanishing on $X$, vanishes identically.

Construct $\widetilde{H}$ as in Lemma~\ref{lemma:construct}.
To prove the theorem, we need to show that a polynomial vanishing on
$\overline{\widetilde{H}^*}$ is the zero polynomial.

Fix two real distinct numbers $\theta_1$ and $\theta_2$
and two positive real numbers $r_1$ and $r_2$.
Then the set $\overline{\widetilde{H}^*} \cap \{ z \in \C^3 :
z_1 = r_1 e^{i\theta_1},
z_2 = r_2 e^{i\theta_2} \}$
contains the set
\begin{equation}
X' = \{ z \in \C^3 :
z_0 = r_1 e^{i\theta_1} x + r_2 e^{i\theta_2} y \text{ where $(x,y) \in
X$} ,
z_1 = r_1 e^{i\theta_1},
z_2 = r_2 e^{i\theta_2} 
 \} .
\end{equation}
$X'$ is a linear image of $X$ under a nonsingular linear mapping.
As $X$ is nonalgebraic then $X'$ is nonalgebraic, and hence
any real polynomial $P$ that vanishes on $\widetilde{H}^*$ vanishes on $X'$
and hence vanishes on the entire set
$\{ z \in \C^3 : z_1 = r_1 e^{i\theta_1}, z_2 = r_2 e^{i\theta_2} \}$.
By varying $\theta_1$ and $\theta_2$ we get that $P$ vanishes on an open
subset of $\C^3$ and thus vanishes identically.
\end{proof}

In the proof of the lemma,
note that not all of $\Delta$ was needed, it
was only used to cover the points where the image of the map could change
depending on what neighbourhood of the complexified variables we took in the
local argument above.
In fact $\Delta$ is an entire Levi-flat hypersurface.  We only need to
take the smallest subvariety that includes $\widetilde{H}
\setminus \Delta$.  Then the subvariety induced by $X$ is the set
$\{ z \in \C^3 : z_0 = z_1 x + z_2 y \text{ where $(x,y) \in X$} \}$ together
with a subset of $\Delta$. 

\begin{cor}
There exists a singular Levi-flat hypersurface $H \subset \bP^2$
that is a semialgebraic set, but is not real-algebraic.  That is, there
exists a real-algebraic Levi-flat subvariety of codimension 1 in $\bP^2$ that is
irreducible as a real-algebraic subvariety, but reducible as a real-analytic
subvariety into two components of codimension 1.
\end{cor}

\begin{proof}
The key is to find an algebraic curve with such properties in $\R^2$
and apply Lemma~\ref{lemma:construct}.  For example, the curve given by
\begin{equation}
x(x-1)(x-2)(x-3) + y^2 = 0
\end{equation}
has two components as a real-analytic subvariety, but is irreducible as a
real-algebraic subvariety.  Taking only one of those components we construct
our $\widetilde{H}$.  This is clearly a proper subset of the algebraic set that
contains $\widetilde{H}$, which must contain both components.  To construct the
full algebraic set we simply solve
$z_0 = z_1x + z_2y$,
$\bar{z}_0 = \bar{z}_1x + \bar{z}_2y$,
for $x$ and $y$ in terms of rational
functions of $z$ and $\bar{z}$, which we plug into the equation
$x(x-1)(x-2)(x-3) + y^2 = 0$ and clear denominators.
\end{proof}

The proof that the construction in this note produces a real-analytic
subvariety depends on the dimension $n=2$.  It is reasonable to conjecture
that a similar construction ought to produce a subvariety in higher
dimension.

Finally, let us remark
why the Levi-foliation does not extend
for hypersurfaces constructed above, a fact already clear from
the theorem in~\cite{Lebl:projlf}, but let us identify explicitly the points
where no extension exists.
Without loss of generality and
possibly translating $X$, we assume that for all small enough $x \in \R$,
there exist at least two distinct $y \in \R$ such that $(x,y) \in X$.
Given any real $s_2$, then for all sufficiently small
real $s_0$, there are at least two distinct leaves of the hypersurface $H$
passing through $[s_0,1,s_2]$.
These points form 
a generic totally-real submanifold, biholomorphic to $\R^2$.
This set would have to lie in the
singular set of any holomorphic foliation tangent to $H$ (that is,
extending the Levi-foliation) near
$[0,1,0]$, which is impossible, as the singular set of a holomorphic foliation
is a complex analytic subvariety.
\def\MR#1{\relax\ifhmode\unskip\spacefactor3000 \space\fi%
  \href{http://www.ams.org/mathscinet-getitem?mr=#1}{MR#1}}

\end{document}